\documentclass[11pt]{article}
\usepackage{amsfonts, amssymb, latexsym, epsfig,mathtools,fullpage}
\usepackage{amsthm}
\usepackage{amsmath}
\usepackage{setspace}
\usepackage{graphicx}
\usepackage[usenames, dvipsnames]{xcolor}
\usepackage{enumerate}
\usepackage{ytableau}
\usepackage{caption}
\usepackage{subcaption}

\usepackage{graphicx}
\usepackage{tikz}
\tikzset{partition/.style={fill,circle,inner sep=1pt}}
\usetikzlibrary{positioning}
\usetikzlibrary{decorations.pathreplacing}
\usetikzlibrary{matrix,arrows}
\usetikzlibrary{calc}
\tikzset{partition/.style={fill,circle,inner sep=1pt},
         part/.style={baseline=0,scale=0.5,bend left=45},
         partlabel/.style={below}}
\usetikzlibrary{shapes,arrows}
\tikzstyle{pnt}=[draw,ellipse,fill,inner sep=1pt]
\tikzstyle{opnt}=[ ]
\tikzstyle{pntt}=[draw,ellipse,fill,inner sep=0.5pt]
\tikzstyle{point}=[draw,ellipse,fill,inner sep=2pt]
\usepackage{color}
\usepackage{hyperref}

\newtheorem{theorem}{Theorem}[section]
\newtheorem{lemma}[theorem]{Lemma}

\theoremstyle{definition}
\newtheorem{definition}{Definition}[section]
\newtheorem{fact}{Fact}[section]
\newtheorem{notation}{Notation}[section]
\newtheorem{remark}{Remark}[section]

\newtheorem{example}[theorem]{Example}

\newcommand{\CC}{\mathbb{C}}
\newcommand{\ZZ}{\mathbb{Z}}

\newcommand{\OO}{\mathcal{O}}
\newcommand{\ssyt}{\mathrm{SSYT}}
\newcommand{\syt}{\mathrm{SYT}}
\newcommand{\pro}{\mathcal{P}}
\newcommand{\evac}{\mathcal{E}}

\newcommand{\inc}[2]{\mathrm{Inc}_#2(#1)}

\DeclareMathOperator{\rd}{read}
\DeclareMathOperator{\ins}{ins}
\DeclareMathOperator{\dis}{Dist}

\DeclareMathOperator{\out}{Out}
\DeclareMathOperator{\inn}{In}
\DeclareMathOperator{\rank}{rank}

\newcommand{\makeset}[2]{ \{#1\,:\, #2\} }
\newcommand{\du}{\;\sqcup\;}

\begin{document}
\title{Proofs and generalizations of a homomesy conjecture of \\Propp and Roby}

\author{
  Jonathan Bloom\\
  \texttt{Lafayette College}\\
  \texttt{bloomjs@lafayette.edu}\\
  \and
  Oliver Pechenik\\
  \texttt{University of Illinois at Urbana-Champaign}\\
    \texttt{pecheni2@illinois.edu}\\
  \and 
  Dan Saracino\\
    \texttt{Colgate University}\\
    \texttt{dsaracino@colgate.edu}
}

\maketitle
\begin{abstract}
Let $G$ be a group acting on a set $X$ of combinatorial objects, with finite orbits, and consider a statistic $\xi : X \to \CC$. Propp and Roby defined the triple $(X, G, \xi)$ to be \emph{homomesic} if for any orbits $\OO_1, \OO_2$, the average value of the statistic $\xi$ is the same, that is \[\frac{1}{{|\OO_1|}}\sum_{x \in \OO_1} \xi(x) = \frac{1}{|\OO_2|}\sum_{y \in \OO_2} \xi(y).\]

In 2013 Propp and Roby conjectured the following instance of homomesy. Let $\ssyt_k(m \times n)$ denote the set of semistandard Young tableaux of shape $m \times n$ with entries bounded by $k$. Let $S$ be any set of boxes in the $m \times n$ rectangle fixed under $180^\circ$ rotation. For $T \in \ssyt_k(m \times n)$, define $\sigma_S(T)$ to be the sum of the entries of $T$ in the boxes of $S$. Let $\langle \pro \rangle$ be a cyclic group of order $k$ where $\pro$ acts on $\ssyt_k(m \times n)$ by promotion. Then $(\ssyt_k(m \times n), \langle \pro \rangle, \sigma_S)$ is homomesic.

We prove this conjecture, as well as a generalization to cominuscule posets. We also discuss analogous questions for tableaux with strictly increasing rows and columns under the K-promotion of Thomas and Yong, and prove limited results in that direction.
\end{abstract}
\ytableausetup{boxsize=1.1em}

\section{Introduction}

Let $G$ be a group acting on a set $X$ of combinatorial objects, with finite orbits, and  $\xi : X \to \CC$ any complex-valued function. The triple $(X, G, \xi)$ is \emph{homomesic} if for any orbits $\OO_1, \OO_2$, the average value of the statistic $\xi$ is the same, that is \[\frac{1}{|\OO_1|}\sum_{x \in \OO_1} \xi(x) = \frac{1}{|\OO_2|} \sum_{y \in \OO_2} \xi(y).\] If $X$ is finite, this implies that the average value of $\xi$ on any orbit is the average value of $\xi$ on $X$. The concept of homomesy was first isolated by Propp and Roby \cite{propp.roby:fpsac, propp.roby:arxiv}, although an example of homomesy was previously conjectured in \cite{panyushev} and proved in \cite{armstrong.stump.thomas}.

Let $\ssyt_k(m \times n)$ denote the set of semistandard Young tableaux of shape $m \times n$, i.e., $m$ rows and $n$ columns, with entries bounded above by $k$. We orient our tableaux in matrix coordinates (English notation).  For $T \in \ssyt_k(m \times n)$ and $S$ a set of boxes in the $m \times n$ rectangle, define $\sigma_S(T)$ to be the sum of the entries of $T$ in the boxes of $S$. Further let $\pro$ be the promotion operator and $\langle \pro \rangle$ be the cyclic group generated by $\pro$.  With this language our main result, which proves a conjecture of Propp and Roby \cite{propp.roby:dartmouth}, is the following.  

\begin{theorem}\label{thm:P--R}
If $S$ is fixed under $180^\circ$ rotation, then $(\ssyt_k(m \times n), \langle \pro \rangle, \sigma_S)$ is homomesic.
\end{theorem}

This result looks remarkably similar to certain homomesies discovered by Propp and Roby \cite{propp.roby:fpsac, propp.roby:arxiv}. These latter results relate to rowmotion (a.k.a.~Fon-der-Flaass map, Panyushev complementation, etc.)~and promotion of order ideals in rectangular posets ${\bf m} \times {\bf n}$. Note that `promotion' in this order ideal context is quite different from the promotion we use for tableaux; the coincidence of terminology reflects the relation worked out in \cite{striker.williams} between tableau promotion for 2-row rectangles and order ideal promotion in the type-A positive root poset. 
Recently Einstein and Propp \cite{einstein.propp} have shown that tableau promotion on $\ssyt_k(m \times n)$ is naturally conjugate to a piecewise-linear lift of order ideal promotion to the rational points with denominator dividing $n$ in the order polytope of ${\bf m} \times {\bf (k - m)}$.
We do not know any concrete relation between Theorem~\ref{thm:P--R} and any of the results of \cite{propp.roby:fpsac, propp.roby:arxiv}.

We structure this paper as follows. In Section~\ref{sec:basics} we define promotion and evacuation for semistandard Young tableaux and prove various important properties. Although most of these results for \emph{standard} Young tableaux may be readily found in the literature (see e.g.~\cite{stanley}), analogous statements and proofs for semistandard Young tableaux are hard to find, if not completely missing from the literature. In Section~\ref{sec:main_result} we prove Theorem~\ref{thm:P--R} via growth diagrams. In Section~\ref{sec:jdt} we give an alternative proof by jeu de taquin in the standard Young tableau case.  In Section~\ref{sec:cominuscule} we prove a generalization to cominiscule posets. In Section~\ref{sec:K} we consider a natural generalization to increasing tableaux under K-promotion and prove the 2-row rectangular case, while identifying counterexamples for general shapes.

\section{Basic facts about promotion and evacuation}\label{sec:basics}
Both promotion and evacuation have been extensively studied in the context of standard Young tableaux (SYT). See~\cite{stanley} for a comprehensive survey. It has been widely believed that most results about promotion and evacuation generalize to the semistandard setting; however, explicit statements and proofs have been mostly lacking from the literature. The purpose of this section is to provide explicit definitions of promotion and evacuation for semistandard Young tableaux, and to prove some of their most important combinatorial properties.

For partitions $\mu\subset \lambda$, we denote by $\ssyt_k(\lambda/\mu)$  the set of all semistandard Young tableaux of skew shape $\lambda/\mu$ with \emph{ceiling} $k$, i.e., all entries are $\leq k$.  If $\mu = \emptyset$ we simply write $\ssyt_k(\lambda)$ and refer to these as \emph{straight-shapes}. Moreover, if $\lambda$ is an $n\times m$ rectangle, then we write $\ssyt_k(n\times m)$ instead of $\ssyt_k(\lambda)$.  

\subsection{Promotion}
For the convenience  of the reader we recall the definitions of \emph{jeu de taquin} and \emph{rectification} which will be needed to define \emph{promotion}.  

Let $T\in \ssyt_k(\lambda/\mu)$.  We say a box in the shape $\mu$ is an \emph{inner corner} of $T$ provided that the boxes immediately below and to the right of it are not in $\mu$. (Consequently, if $T$ is a straight-shape it has no inner corners.)  Beginning with any  inner corner $b_0$ we define the (unique) sequence of boxes $b_0, b_1, \ldots, b_m$ so that $b_{i+1}$ is whichever of the boxes immediately below or to the right of $b_i$  contains the smaller value. In the case where these two boxes contain equal values $b_{i+1}$ is chosen to be the box below $b_i$. In the case where exactly one of these boxes lies outside $\lambda$ the other is chosen.  This process continues until we reach a box $b_m$ such that the boxes below it and to its right fall outside $\lambda$.   Using this sequence $b_0, b_1, \ldots, b_m$, we obtain a new tableau by sliding the value in $b_{i+1}$ into $b_i$.  It is clear from this construction that $b_0$ is contained in the resulting tableau. We refer to this construction as a \emph{jeu de taquin slide}. A straightforward argument shows that the resulting tableau is also semistandard.

Since a jeu de taquin slide ``removes" an inner corner, it follows that starting with $T$ and iteratively performing such slides until all inner corners have been removed will result in a  straight-shaped semistandard tableau.  This iterative process of obtaining a straight-shaped tableau from a skew-shaped one is called \emph{jeu de taquin}.   It is a classic theorem (see~\cite[p.~15]{fulton}) that the resulting straight-shaped tableau is independent of the order in which the  jeu de tauqin slides are performed.  Consequently, there is a unique straight-shaped tableau obtained by performing jeu de taquin on $T$.  This straight-shaped tableau is called the  \emph{rectification} of $T$.

For the remainder of this section we will fix an arbitrary partition $\lambda$.  We are now in a position to define promotion.

\begin{definition}\label{def:promotion}
Let $T\in \ssyt_k(\lambda)$.  The \emph{promotion} $\pro(T)$ of $T$  is given by the following construction.  If $T$ has no 1's then let $\pro(T)$ be the result of decrementing all the values of $T$ by 1.  Otherwise, do the following.  First, delete all the boxes in $T$ that contain a 1 and rectify the resulting skew tableau.  Next, decrement all the values by 1 and then place a $k$ in all the empty boxes resulting from sliding (so that $\pro(T)$ has the same shape as $T$). 
\end{definition}

\begin{example}\label{ex:promotion}

Let $k=6$ and $T =  \ytableaushort{1123,3344,55}\, .$ Then $\pro(T) = \ytableaushort{1223,2366,44}\, .$ \qed
\end{example}

We will also need an alternative definition of promotion based on Bender-Knuth operations, which, following Striker and Williams~\cite{striker.williams}, we  call \emph{toggles} for short.  For a fixed $1\leq i\leq k$, define the toggle $\tau_i$ as follows.  For a given row of $T\in \ssyt_k(\lambda)$ locate all the boxes that contain either an $i$ or an $i+1$ and do not have any other $i$'s or $(i+1)$'s in their columns.  Suppose the number of such $i$'s is $a$ and the number of such $(i+1)$'s is $b$.  Now replace these $i$'s and $(i+1)$'s by $b$ $i$'s and $a$ $(i+1)$'s so that the resulting tableau is semistandard.  Repeat the process for every row of $T$ and denote by $\tau_i(T)$ the resulting semistandard Young tableau.  Observe that we may perform the individual row operations in any order; hence $\tau_i(T)$ is well defined.  

The alternative characterization of promotion is as follows.   

\begin{theorem}\label{thm:toggle}
For any $T\in \ssyt_k(\lambda)$, we have $\pro(T) = \tau_{k-1} \circ \tau_{k-2} \circ \cdots \circ\tau_1(T)$.
\end{theorem}

Before we discuss and prove this theorem, we give an example.

\begin{example}\label{ex:promotion_toggle}
Using the tableau $T$ from Example~\ref{ex:promotion} we obtain

\begin{align*} 
T = \ytableaushort{1123,3344,55}&\xrightarrow{\quad\tau_1\quad}\ytableaushort{1223,3344,55}\xrightarrow{\quad\tau_2\quad}\ytableaushort{1223,2344,55}\xrightarrow{\quad\tau_3\quad}\ytableaushort{1223,2344,55}\\
&\xrightarrow{\quad\tau_4\quad}\ytableaushort{1223,2355,44}\xrightarrow{\quad\tau_5\quad}\ytableaushort{1223,2366,44}\, ,
\end{align*}
which we see is the same as $\pro(T)$.  
\qed
\end{example} 

A proof of Theorem~\ref{thm:toggle} appears as a special case of Lemma 5.2 in Gansner's 1980 paper~\cite{gansner}, although the result is easy to miss there since the paper focuses more on evacuation (and does not use either of the terms ``promotion" or ``evacuation").  For convenience, we present here a short proof of Theorem~\ref{thm:toggle}.

We begin by reformulating the Bender-Knuth toggle $\tau_k(T)$ for $T\in \ssyt_{k+1}(\lambda)$.  

\begin{definition}
Let $T\in\ssyt_{k+1}(\lambda)$.  We define $t_k(T)\in\ssyt_{k+1}(\lambda)$ by the following 3-step construction.  Step 1 is to delete all the $k$'s from $T$.  For step 2, identify all the $(k+1)$'s with an empty box directly above.  Slide these $(k+1)$'s up one unit into these empty boxes.  Then, slide all the remaining $(k+1)$'s as far left as possible.   For step 3, decrement all the $(k+1)$'s by 1 and place $(k+1)$'s in all the empty boxes so that the resulting tableau $t_k(T)\in \ssyt_{k+1}(\lambda)$.
\end{definition}

It is not hard to check that $\tau_k(T) = t_k(T)$ for $T\in\ssyt_{k+1}(\lambda)$.

\begin{example}\label{ex:t_k}
To illustrate the steps in our definition of $t_k$, consider $k=3$ and 
$$R = \ytableaushort{1123333344,23444,3,4}\,.$$ 
Starting with $R$, the 3 steps in the construction of $t_3(R)$ are
$$ \ytableaushort{112{*(white)}{*(white)}{*(white)}{*(white)}{*(white)}44,2{*(white)}444,{*(white)},4}\quad\longrightarrow\quad \ytableaushort{1124444{*(white)}{*(white)}{*(white)},24{*(white)}{*(white)}{*(white)},4,{*(white)}}\quad\longrightarrow\quad  \ytableaushort{1123333444,23444,3,4}\, .$$ 
Note that $t_3(R) = \tau_3(R)$.  \qed
\end{example}

We will also need a slight modification of the promotion operator $\pro$.  

\begin{definition}
We define  $\pro_i(T)$ to be the result of freezing all the elements of $T$ which are at least $i+1$ and then promoting (in the sense of Definition~\ref{def:promotion}) the unfrozen elements with ceiling $i$.  
\end{definition}

To prove Theorem~\ref{thm:toggle}, it will now suffice to prove the first equality in the following lemma.

\begin{lemma}\label{lem:inductive_step}
Let $T \in \ssyt_{k+1}$.  Then 
$$\pro(T) = t_k \circ \pro_k(T) = \tau_k \circ \pro_k(T)$$
\end{lemma}

In order to prove the first equality in Lemma~\ref{lem:inductive_step} we need the following technical result. 

\begin{lemma}\label{lem:noncrossing}
Let $T$ be any semistandard skew tableau and let $T_1$ be the result of a jeu de taquin slide applied to $T$, starting at inner corner $I_1$.  Likewise, let  $T_2$ be the result of a  jeu de taquin slide applied to $T_1$, starting at inner corner $I_2$, where $I_2$ is on the same row as $I_1$.  Then on any row of $T$ that contains boxes involved in the first slide and boxes involved in the second slide, the rightmost box involved in the second slide is strictly to the left of the rightmost box involved in the first slide.  
\end{lemma}

\begin{proof}
By a simple induction on rows in $T$ we see that it will suffice to prove the following claim.  Let $b_1,b_2, b_3,b_4$ be boxes in $T$ with values $x,y,z,$ and $w$, respectively,  so that together they form a 2$\times$2 contiguous square in $T$ as follows  
$$\ytableaushort{xy,zw}\ .$$
If the first jeu de taquin slide involves the boxes $b_2$ and $b_4$ and the second slide involves the box $b_1$, then the second slide must also involve the box $b_3$.  

To prove this claim, first note that as $T$ is semistandard then $z\leq w$. Next, observe that since our first slide involves boxes $b_2$ and $b_4$, the box $b_2$ after the first slide must contain $w$.  Additionally, it follows that $b_3$ is not involved in the first slide and therefore in $T_1$ box $b_3$ still contains $z$ .  As the second slide involves $b_1$, it follows that it must also involve $b_3$.  
\end{proof}

\begin{proof}[Proof of Lemma~\ref{lem:inductive_step}]
Let $T^-$ be the skew tableau obtained by deleting all the boxes in $T$ that contain a 1.  It will suffice to prove that the rectification of  $T^-$ is equivalent to the following two-step ``rectification".  (To help the reader we illustrate this equivalence in Example~\ref{ex:rect}.)   First, freeze all the $(k+1)$'s in $T^-$ and rectify the unfrozen boxes.  Observe that some of the $(k+1)$'s in the resulting object will have empty boxes above them or to their left.  Label these empty boxes, from right to left, $b_1,b_2,\ldots$  Next, slide the $(k+1)$'s into these empty boxes as described in step 2 of the definition of $t_k$.

To see the equivalence, note that during the full rectification of $T^-$ the $(k+1)$'s slide into (some of) the empty boxes $b_1,b_2,\ldots$. A priori there might be some ambiguity as to whether a given $(k+1)$ slides left or up.  (To see this, consider a $k+1$ such that in the two-step ``rectification" process this $k+1$ has empty boxes above and to its left.) Applying Lemma~\ref{lem:noncrossing} to the rectification of $T$ with $k+1$ frozen, we see that preference in such a case is always given to sliding up.   It now follows that regardless of which construction is used, the final arrangement of the $(k+1)$'s is identical.  
\end{proof}

\begin{example}\label{ex:rect}
If $k = 3$ and 
$$T = \ytableaushort{1111111344,22444,3,4}$$ 
then the rectification of $T^-$ is 
\begin{equation}\label{eq:1}
\ytableaushort{2234444,34,4}\, .
\end{equation}
On the other hand if we first freeze all the $4$'s in $T^-$ and then rectify the unfrozen boxes we obtain
$$\ytableaushort{223{*(white)}{*(white)}{*(white)}{*(white)}44,3{*(white)}444,{*(white)},4} \, .$$
Sliding the $4$'s gives the tableau in~(\ref{eq:1}), as claimed. \qed
\end{example}

\subsection{Evacuation}

We now define evacuation for semistandard Young tableaux.

\begin{definition}
For $T\in\ssyt_k(\lambda)$, define a sequence $\epsilon_1(T), \epsilon_2(T),\ldots, \epsilon_k(T)$ as follows.  Let $\epsilon_1(T) = \pro(T)$.  For $j\geq 2$, obtain $\epsilon_j(T)$ by freezing the entries $k,\ldots, k-(j -2)$ in $\epsilon_{j-1}(T)$ and promoting the remaining portion.
We define the \emph{evacuation} $\evac(T)$ of $T$ to be $\epsilon_k(T)$.  
\end{definition}

Using the Bender-Knuth toggle characterization of promotion, we see that evacuation has the following alternative description:
$$\evac = \tau_1\cdot (\tau_2\tau_1)\cdot\,\cdots\,\cdot (\tau_{k-3}\cdots \tau_1)\cdot (\tau_{k-2}\cdots \tau_1)\cdot (\tau_{k-1}\cdots \tau_1).$$
\begin{notation}

For rectangular $T\in \ssyt_k(m\times n),$ let $T^+$ denote the element of $\ssyt_k(m\times n)$ obtained by rotating $T$ by $180^\circ$ and then replacing each entry $i$ by $k+1-i.$
\end{notation}

In the context of  rectangular semistandard Young tableaux, we will also need the \emph{dual evacuation} of $T$, which we denote by $\evac'(T)$. This is defined analogously to evacuation except that here we use the inverse of promotion and freeze elements from smallest to largest.  It is easy to see that 
\begin{equation}\label{eq:dual_evac}
\evac'(T)=\evac(T^+)^+.
\end{equation}  
As above dual evacuation also has a characterization in terms of toggles.  Explicitly it is
$$\evac' = \tau_{k-1}\cdot (\tau_{k-2}\tau_{k-1})\cdot\,\cdots\,\cdot (\tau_3\cdots \tau_{k-2}\tau_{k-1})\cdot (\tau_2\cdots \tau_{k-2}\tau_{k-1})\cdot (\tau_1\cdots \tau_{k-2}\tau_{k-1}).$$

\begin{example}\label{ex:evacuation}
Using the $T$ from Example~\ref{ex:promotion}, we illustrate each step in the definition of evacuation below. The shading at each step denotes the boxes that are frozen.

\begin{align*} 
T=\ytableaushort{1123,3344,55}&\longrightarrow\ytableaushort{1223,23{*(gray)6}{*(gray)6},44}\longrightarrow\ytableaushort{1112,23{*(gray)6}{*(gray)6},3{*(gray)5}}\longrightarrow\ytableaushort{11{*(gray)4}{*(gray)4},22{*(gray)6}{*(gray)6},{*(gray)4}{*(gray)5}}\longrightarrow\ytableaushort{11{*(gray)4}{*(gray)4},33{*(gray)6}{*(gray)6},{*(gray)4}{*(gray)5}}\\
&\longrightarrow\ytableaushort{11{*(gray)4}{*(gray)4},{*(gray)3}{*(gray)3}{*(gray)6}{*(gray)6},{*(gray)4}{*(gray)5}}\longrightarrow\ytableaushort{{*(gray)2}{*(gray)2}{*(gray)4}{*(gray)4},{*(gray)3}{*(gray)3}{*(gray)6}{*(gray)6},{*(gray)4}{*(gray)5}},
\end{align*}
So 
$$\evac(T) = \ytableaushort{2244,3366,45}\, .$$ \qed
\end{example}

\subsection{A fundamental theorem on promotion and evacuation}
The following theorem contains the results we will need about promotion and evacuation.  For the special case of standard tableaux, proofs of parts (a), (b), and (c) are readily available in the literature (see, e.g., \cite[Theorem 2.1]{stanley}) and are essentially due to Sch\"utzenberger.

\begin{theorem}\label{thm:coxeter_rltn}
Let $T \in \ssyt_k(\lambda)$. Then 
\begin{enumerate}[(a)]
\item $\evac^2(T) = T$,
\item $\evac \circ \pro(T) = \pro^{-1} \circ \evac(T)$,
\item if $\lambda$ is rectangular, $\pro^{k}(T) = T$,
\item if $\lambda$ is rectangular, $\evac(T)=T^+.$
\end{enumerate}
\end{theorem}

\begin{proof}[Proof of parts (a)-(c)]
We take part (d) as given. (Part (d) is proved below without reference to (a)--(c).)  
We imitate the proof of \cite[Theorem 2.1]{stanley} (based on an idea of Haiman~\cite{haiman}), using the formulation of promotion in terms of Bender-Knuth toggles.
 
Let $G$ be the group with generators $x_1,\ldots, x_{k-1}$ and relations
\begin{equation}
        \begin{aligned}
        x_i^2 &= 1,\quad 1\leq i\leq k-1\\
        x_ix_j &= x_jx_i, \quad \textrm{if } |i-j|>1.
        \end{aligned}
\label{eq:relations}
\end{equation}
Let
\begin{equation*}
        \begin{aligned}
        y&=x_{k-1}x_{k-2}\cdots x_1\\
        z&=x_1\cdot (x_2x_1)\cdot\,\cdots\,\cdot (x_{k-3}\cdots x_1)\cdot (x_{k-2}\cdots x_1)\cdot (x_{k-1}\cdots x_1)\\
        z'&=x_{k-1}\cdot (x_{k-2}x_{k-1})\cdot\,\cdots\,\cdot (x_3\cdots x_{k-2}x_{k-1})\cdot (x_2\cdots x_{k-2}x_{k-1})\cdot (x_1\cdots x_{k-2}x_{k-1}).
        \end{aligned}
\end{equation*}
By \cite[Lemma 2.2]{stanley}, the following hold in $G$:
\begin{equation}
       \begin{aligned}
       z^2&=(z')^2=1\\
       y^k&=z'z\\
       zy&=y^{-1}z.
       \end{aligned}
\label{eq:laws}
\end{equation}
Since the Bender-Knuth  toggles $\tau_1,\ldots, \tau_{k-1}$ satisfy the defining relations (\ref{eq:relations}) of $G$, the equations in (\ref{eq:laws}) hold after we replace $x_i$ by $\tau_i$.  After these replacements, $y$ becomes $\pro$ and $z$ becomes $\evac$.  This proves parts (a) and (b) of the theorem.

Now assume $T\in\ssyt_k(m\times n)$ is rectangular.  In this case $z'$ becomes $\evac'$, which immediately yields $\pro^k=\evac'\circ\evac$.  Additionally, 
$$\evac'\circ\evac(T) = \evac'(T^+) = \evac(T)^+ = T,$$
where the first and third equalities follow from part (d) and the second equality is just Equation~\ref{eq:dual_evac}. This proves part (c).
\end{proof}

For the proof of part (d) we will need the basic language of reading words and the Robinson-Schensted-Knuth (RSK) correspondence in the context of SSYT.  Those unfamiliar with these ideas may consult \cite[Chapters 1--4]{fulton}.  

To set notation we let $\ins (w)$ be the insertion tableau  of a word $w$ under RSK, and for any tableau $P$ we let $\rd(P)$ denote its reading word. The main ingredient for our proof of part (d) is the following standard fact, which is a special case of part 4 of the Duality Theorem of \cite[p.~184]{fulton}.   

\begin{fact}\label{thm:reverse}
Fix $k>0$ and let $w=w_1\cdots w_n$ be a word in the letters $\{1,\ldots,k\}$ and $w^+ = (k+1-w_n)(k+1-w_{n-1})\ldots (k+1-w_1)$.  If $\ins(w) = P$, then $\ins(w^+) = \evac(P)$.
\end{fact}

Armed with this fact we now prove part (d).  

\begin{proof}[Proof of part (d)]
As $T$ is rectangular, we have $\rd(T^+) = \rd(T)^+$.   Combining this observation with Fact~\ref{thm:reverse} yields
$$T^+ = \ins(\rd(T^+)) = \ins(\rd(T)^+) = \evac(\ins(\rd(T))) = \evac(T),$$
where the first and last equalities are the standard fact that the insertion tableau of a reading word is just the underlying tableau.  
\end{proof}

\section{Proof of the main result by growth diagrams}\label{sec:main_result}

\begin{definition}
Let $T\in \ssyt_k(m\times n).$  For a box $B$ in $T$, we define $\dis(T,B)$ to be the multiset 
$$\dis(T,B) =\makeset{\sigma_{\{B\}}(\pro^i(T))}{0\leq i\leq k-1}.$$
\end{definition}

\begin{lemma}\label{lem:main_gd}
If $T\in \ssyt_k(m\times n)$ and $B$ is a box in $T$, then  $\dis(T,B) =\dis(\evac(T), B)$.  
\end{lemma}
We delay the proof of Lemma~\ref{lem:main_gd}, first showing how Theorem~\ref{thm:P--R} follows immediately. 
\begin{proof}[Proof of Theorem~\ref{thm:P--R}]
If $T\in \ssyt_k(m\times n)$, then it follows from Theorem~\ref{thm:coxeter_rltn}(b, c), that the orbits of $T$ and $\evac(T)$ under promotion are of the same size $\ell$, and that $\ell | k$.  By Lemma~\ref{lem:main_gd} and Theorem~\ref{thm:coxeter_rltn}(b) we have the following multiset equalities
$$\makeset{\sigma_{\{B\}}(\pro^i( T))}{0\leq i<\ell} = \makeset{\sigma_{\{B\}}(\pro^i \circ\evac( T))}{0\leq i<\ell} = \makeset{\sigma_{\{B\}}(\evac\circ\pro^i( T))}{0\leq i<\ell}.$$
Theorem~\ref{thm:coxeter_rltn}(d) now implies that $\dis(T,B) = \makeset{k+1-i}{i\in \dis(T,B^*) }$, where $B^*$ is the box corresponding to $B$ under $180^\circ$ rotation. This last statement immediately implies Theorem~\ref{thm:P--R}.  Specifically, the average value of $\sigma_S$ on any orbit is $\frac{(k+1)|S|}{2}$.
\end{proof}

The remainder of this section is devoted to a proof of Lemma~\ref{lem:main_gd}, using the growth diagrams of S.~Fomin. (For additional information on growth diagrams, cf.~\cite[Appendix~1]{ec} or \cite[$\mathsection 5$]{stanley}.) For $T \in \ssyt_k(\lambda)$, the \emph{growth diagram} of $T$ is built as follows. Let $T_{\leq j}$ denote the Ferrers diagram consisting of those boxes of $T$ with entry $i \leq j$. Identify $T$ with a particular multichain in the Young lattice, explicitly with the sequence of Ferrers diagrams $(T_{\leq j})_{0 \leq j \leq k}$. Note that this sequence uniquely encodes $T$. Now write this sequence of Ferrers diagrams horizontally from left to right. Below this sequence, draw, in successive rows, the sequences of Ferrers diagrams associated to $\pro^i(T)$ for $i \geq 1$. Above this sequence, draw, in successive rows, the sequences of Ferrers diagrams associated to $\pro^i(T)$ for $i \leq -1$. This gives a doubly infinite array.  Now offset each row one position to the right of the row immediately above it. Example~\ref{ex:growth_diagram} shows an example of this construction. The \emph{rank} of a partition in the growth diagram is the number of partitions appearing strictly left of it in its row, or equivalently the number of partitions appearing strictly below it in its column.

\begin{example}\label{ex:growth_diagram}\ 
Let $T \in \ssyt_5(2 \times 3)$ be the semistandard Young tableau \ytableaushort{123,344}\  . Then the growth diagram of $T$ is
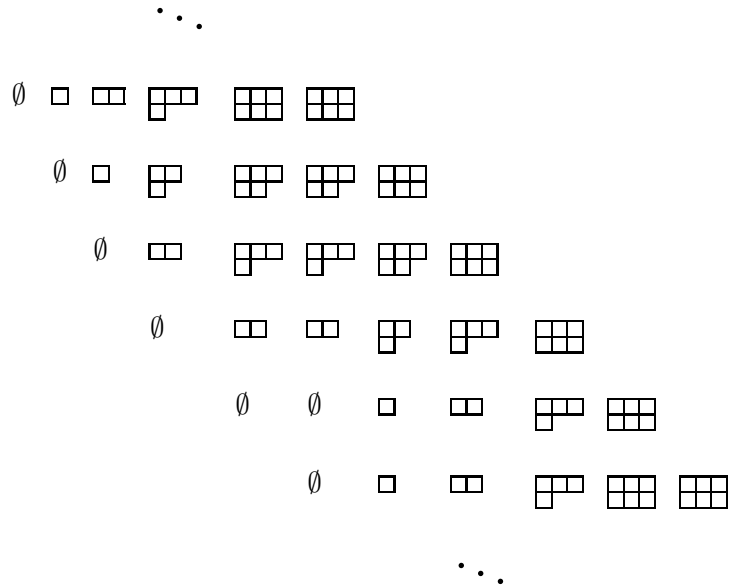
\ytableausetup{boxsize=.5em}
\[\begin{array}{lllllllllllll} 
& & & {\Huge \ddots} \\ \\
\emptyset & \ydiagram{1} & \ydiagram{2} & \ydiagram{3,1} & \ydiagram{3,3} & \ydiagram{3,3} \\ \\  
& \emptyset & \ydiagram{1} & \ydiagram{2,1} & \ydiagram{3,2} & \ydiagram{3,2} & \ydiagram{3,3}  \\ \\ 
& & \emptyset & \ydiagram{2} & \ydiagram{3,1} & \ydiagram{3,1} & \ydiagram{3,2} & \ydiagram{3,3} \\ \\ 
& & & \emptyset & \ydiagram{2} & \ydiagram{2} & \ydiagram{2,1} & \ydiagram{3,1} & \ydiagram{3,3} \\ \\ 
& & & & \emptyset & \emptyset & \ydiagram{1} & \ydiagram{2} & \ydiagram{3,1} & \ydiagram{3,3} \\ \\
& & & & & \emptyset & \ydiagram{1} & \ydiagram{2} & \ydiagram{3,1} & \ydiagram{3,3} & \ydiagram{3,3}\\ \\
& & & & & & & {\Huge \ddots} \\ \\ 
\end{array}\] where the top displayed row corresponds to $T$ and the bottom displayed row to $\pro^5(T)=T$. Each row encodes a chain of length $5$, since we consider $T \in \ssyt_5(2 \times 3)$.
\ytableausetup{boxsize=1.1em} \qed
\end{example}

\begin{proof}[Proof of Lemma~\ref{lem:main_gd}]
Let $T$ and $B$ be as in the statement of the lemma, and consider the growth diagram of $T$. We darken all shapes in the growth diagram that contain the box $B$ (as in Example~\ref{ex:shading_and_Dyck_path}). Consider any row and the tableau $R$ it encodes. Now look at the column containing the rightmost Ferrers diagram in this row. It is well known that, for standard $T$, this column is the multichain of shapes that encodes $\evac(R)$.  (See~\cite[p. 427]{ec}.) In fact the same is true for semistandard $T$.  To verify this, we observe that, for $1\leq j\leq k$, the shape with rank $k-j$ in the indicated column is $\pro^j(R)_{\leq k-j}$, and we only need to verify that this is $\evac(R)_{\leq k-j}$, i.e., that $\pro^j(R)_{\leq k-j} = \epsilon_j(R)_{\leq k-j}$. But in fact more than this is true.  The placements of the integers $1,\ldots, k-j$ in $\pro^j(R)_{\leq k-j}$ are exactly the same as the placements of these integers in $\epsilon_j(R)$. This follows from the fact that for any $V$ in $\ssyt_k(\lambda)$, and every positive integer $m\leq k$, the placements of $1,\ldots,m+1$ in $V$ determine the placements of $1,\ldots, m$ in $\pro(V)$.  

It now follows from Theorem~\ref{thm:coxeter_rltn}(b) that if a column of the growth diagram encodes a tableau $V$, then the column to the left of this column encodes $\pro(V)$.  

Returning to the growth diagram of $T$, note that if we fix any set $\{R_i : i \in I\}$ of $k$ consecutive rows, then as multisets $\dis(T,B) = \{ \text{rank}(D_i) : i \in I\}$, where $D_i$ is the leftmost darkened shape in $R_i$. Similarly if we fix any set $\{C_j : j \in J\}$ of $k$ consecutive columns, then $\dis(\evac (T), B) = \{\text{rank}(D_j) : j \in J\}$, where $D_j$ is the bottommost darkened shape in $C_j$.  For ease of exposition we will call a darkened shape \emph{row-minimal} if it is the leftmost darkened shape in some row, and \emph{column-minimal} if it is the bottommost darkened shape in some column.  We call a darkened shape \emph{minimal} if it is either row-minimal or column-minimal.  

To see that $\dis(T,B) = \dis(\evac(T),B)$, let $R_0,\ldots, R_k$ be any set of $k+1$ consecutive rows of the growth diagram in descending order. Let $D_0$ and $D_k$ be the row-minimal shapes in rows $R_0$ and $R_k$, respectively, and note that the column containing $D_k$ is $k$ columns to the right of the column containing $D_0$.  Now list all the minimal shapes in row $R_0$ from left to right, followed by all the minimal shapes in row $R_1$, and so on, concluding with just the single minimal shape $D_k$ from row $R_k$.  Consider all these shapes to be vertices.  Note that two successive vertices $D_j,D_i$ in this list may have the same rank, $r$, if $D_j$ is column-minimal and $D_i$ is row-minimal (in the next row).  Whenever this occurs we insert a new vertex of rank $r+1$ to the right of $D_j$ and above $D_i$.  If the elements of the augmented list of vertices are $v_0,v_1,\ldots, $ we define a path $P$ in the first quadrant of the $xy$-plane by replacing each $v_i$ by the point $(i,\rank(v_i))$, and connecting successive points with up-steps $(1,1)$ and down-steps $(1,-1)$.  

By the preceding paragraph $\dis(T,B)$ is the multiset of ranks of row-minimal shapes in rows $R_1$ through $R_k$.  By the construction of $P$ this is the multiset $M_1$ of heights of right endpoints of down-steps in $P$.  Since $P$ starts and ends at the same height, $M_1$ equals the multiset $M_2$ of heights of left endpoints of up steps in $P$.  By the construction of $P$, $M_2$ is the multiset of ranks of column-minimal shapes in rows $R_0$ through $R_{k-1}$, i.e., $M_2$ is $\dis(\evac(T),B)$.  This concludes the proof.  

\end{proof}

\begin{example}\label{ex:shading_and_Dyck_path}
As in Example~\ref{ex:growth_diagram}, let $$T=  \ytableaushort{12{*(gray)3},344}\; ,$$ where we have shaded the box $B$. If we now shade all the Ferrers diagrams containing $B$, we obtain the following shaded growth diagram:
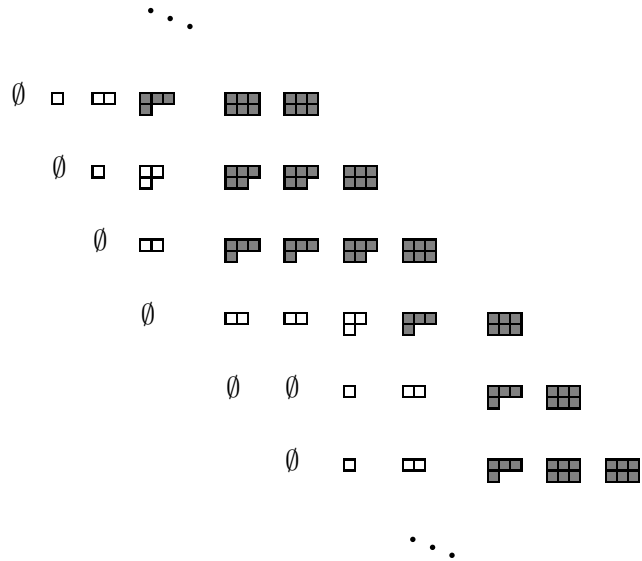
\ytableausetup{boxsize=.35em}
\[\begin{array}{lllllllllllll}
& & & {\Huge \ddots} \\ \\
\emptyset & \ydiagram{1} & \ydiagram{2} & \ydiagram[*(gray)]{3,1} & \ydiagram[*(gray)]{3,3} & \ydiagram[*(gray)]{3,3} \\ \\  
& \emptyset & \ydiagram{1} & \ydiagram{2,1} & \ydiagram[*(gray)]{3,2} & \ydiagram[*(gray)]{3,2} & \ydiagram[*(gray)]{3,3}  \\ \\ 
& & \emptyset & \ydiagram{2} & \ydiagram[*(gray)]{3,1} & \ydiagram[*(gray)]{3,1} & \ydiagram[*(gray)]{3,2} & \ydiagram[*(gray)]{3,3} \\ \\ 
& & & \emptyset & \ydiagram{2} & \ydiagram{2} & \ydiagram{2,1} & \ydiagram[*(gray)]{3,1} & \ydiagram[*(gray)]{3,3} \\ \\ 
& & & & \emptyset & \emptyset & \ydiagram{1} & \ydiagram{2} & \ydiagram[*(gray)]{3,1} & \ydiagram[*(gray)]{3,3} \\ \\
& & & & & \emptyset & \ydiagram{1} & \ydiagram{2} & \ydiagram[*(gray)]{3,1} & \ydiagram[*(gray)]{3,3} & \ydiagram[*(gray)]{3,3}\\ \\
& & & & & & & {\Huge \ddots} \\ \\ 
\end{array}
\] 
We have $\dis(T,B) = \{2,3,3,4,4\} = \dis(\evac(T),B)$, and we obtain the path 
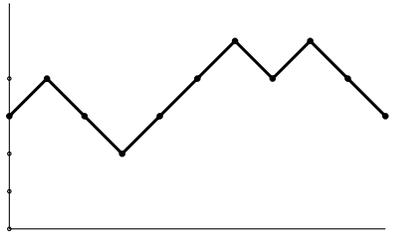
\begin{center}
\begin{tikzpicture}[scale=0.5]
\def\U{-- ++(1,1) [fill] circle(1.3pt)}
\def\D{-- ++(1,-1) [fill] circle(1.3pt)}

\draw[line width = .4mm] (0,3) circle(1.3pt)\U\D\D\U\U\U\D\U\D\D;

\draw(0,0)--(10,0);
\draw(0,0)--(0,6);
\draw (0,0) circle(1.3pt);
\draw (0,1) circle(1.3pt);
\draw (0,2) circle(1.3pt);
\draw (0,3) circle(1.3pt);
\draw (0,4) circle(1.3pt);

\end{tikzpicture}\quad.
\end{center} \qed
\end{example}

\begin{remark}
Note that the same proof shows that Lemma~\ref{lem:main_gd} remains true for $T \in \ssyt_k(\lambda)$ even when $\lambda$ is not rectangular.
\end{remark}

\begin{remark}
Growth diagrams are closely related to the toggles of Section~\ref{sec:basics}. We illustrate with an example. Consider a path through the below growth diagram that starts at the left side and reaches the right by a sequence of `hops', either one Ferrers diagram up or one Ferrers diagram to the right. This path encodes a semistandard tableau in an obvious way. In this example, the solid line encodes the tableau $A=\ytableausetup{boxsize=1.1em} \ytableaushort{113,235}$, while the dotted line encodes $B=\ytableaushort{112,235}$. It follows easily from the definitions, that `bending' the path at a corner (or at either end) corresponds to applying a single toggle operator, $\tau_i$. In this example, $B = \tau_2(A)$ and $A = \tau_2(B)$. Note that this observation gives an alternative way of seeing that the central column encodes the evacuation of the top row.

$
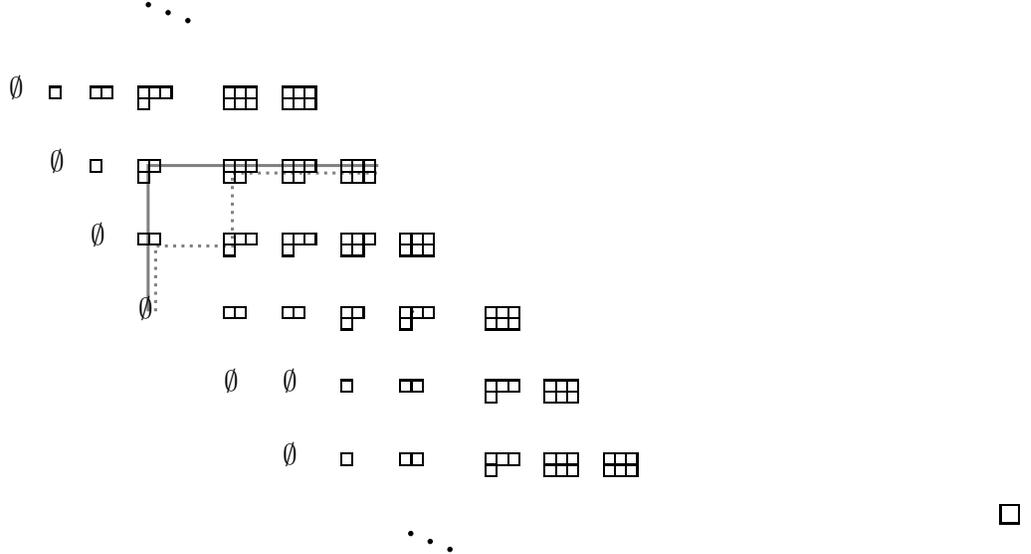
\begin{picture}(24,200)
\put(122,80){$\begin{tikzpicture}
\def\U{-- ++(0,.97) }
\def\D{-- ++(1.02,0) }
\def\SU{-- ++(0,.87) }
\def\SD{-- ++(0.92,0) }

\draw[line width = .4mm, color=gray] (0,3) \U\U\D\D\D;
\draw[line width = .4mm, color=gray, dotted] (0.1,3) \SU\D\U\D\SD;
\end{tikzpicture}\quad.$}
\ytableausetup{boxsize=.35em}
\put(65,85){$\begin{array}{lllllllllllll} 
& & & {\Huge \ddots} \\ \\
\emptyset & \ydiagram{1} & \ydiagram{2} & \ydiagram{3,1} & \ydiagram{3,3} & \ydiagram{3,3} \\ \\  
& \emptyset & \ydiagram{1} & \ydiagram{2,1} & \ydiagram{3,2} & \ydiagram{3,2} & \ydiagram{3,3}  \\ \\ 
& & \emptyset & \ydiagram{2} & \ydiagram{3,1} & \ydiagram{3,1} & \ydiagram{3,2} & \ydiagram{3,3} \\ \\ 
& & & \emptyset & \ydiagram{2} & \ydiagram{2} & \ydiagram{2,1} & \ydiagram{3,1} & \ydiagram{3,3} \\ \\ 
& & & & \emptyset & \emptyset & \ydiagram{1} & \ydiagram{2} & \ydiagram{3,1} & \ydiagram{3,3} \\ \\
& & & & & \emptyset & \ydiagram{1} & \ydiagram{2} & \ydiagram{3,1} & \ydiagram{3,3} & \ydiagram{3,3}\\ \\
& & & & & & & {\Huge \ddots} \\ \\ 
\end{array}$}
\end{picture}$
\qed
\end{remark}
\ytableausetup{boxsize=1.1em}

\section{A jeu de taquin proof of the standard case}\label{sec:jdt}

In this section we present an alternative proof of Lemma~\ref{lem:main_gd} using jeu de taquin in the context of standard young tableaux.  As this lemma is the main ingredient in the proof of Theorem~\ref{thm:P--R}, this section provides an alternative proof of the main result for the special case of standard Young tableaux.  

First we fix some notation.  Let $T\in\syt(m\times n)$, $k=nm$, and $B$ be any box in $T$. We define the \emph{labeled promotion path} of $T$ to be the double sequence
$$\rho(T) = (X_1,\ldots, X_\ell, \alpha_1,\ldots,\alpha_\ell)$$ 
given by the following algorithm.  First let $X_1$ be the box in the upper left corner in $T$.  For $i \geq 1$, recursively define $X_{i+1}$ to be either the box below $X_i$ or to the right of $X_i$ according to which contains the smaller value in $T$. If exactly one of these two boxes does not  exist in $T$, choose the one that exists.  We stop when we reach the lower right corner of $T$;  consequently, $\ell = n+m-1$.  Finally let $\alpha_i$ be the value in box $X_i$.     Now $\pro( T)$ may be defined with respect to $\rho(T)$ as follows.  First delete the value 1 in $X_1$, shift each $\alpha_i$ to box $X_{i-1}$, insert the value $k+1$ into box $X_\ell$, and then decrement all values in $T$ by one.  Likewise, we may define  $\pro^{-1}(T)$  by reversing the above algorithm. It will be helpful to observe that in the mapping $T \mapsto \pro( T)$ the values slide northwest, whereas in the mapping $T\mapsto \pro^{-1}(T)$ the values slide southeast. 

For an example of these definitions see Example~\ref{ex:promotion_path}.

\begin{example}\label{ex:promotion_path}
Consider $k=9$ and
$$T =  \ytableaushort{{*(gray)1}{*(gray)2}5,3{*(gray)4}{*(gray)7},68{*(gray)9}}\, .$$  
The sequence of boxes in $\rho(T)$ is given by the shaded path and the sequence of labels is $1,2,4,7,9$.  As described by the above paragraph we obtain
$$\pro(T)=\ytableaushort{134,268,579}\, .$$
\qed
\end{example}

 Now consider the following progression, which generates the orbit of $T$:
\begin{equation}\label{eq:promotionprogression}
T \longrightarrow \pro( T) \longrightarrow\cdots \longrightarrow \pro^{k-1} (T) \longrightarrow T.
\end{equation}
Setting $(X^j_1,\ldots,X^j_\ell, \alpha^j_1,\ldots,\alpha^j_\ell) = \rho(\pro^j (T))$ for $0\leq j<k$, we define the \emph{multiset}
$$\out(T,B) = \makeset{\alpha^j_i}{X^j_i = B, 0\leq j<k}.$$  
In words, this is the  multiset of all values that slide out of $B$ during the progression~(\ref{eq:promotionprogression}).  Likewise, we define the \emph{multiset}
$$\inn(T,B) = \makeset{\alpha^j_{i+1}-1}{X^j_i = B, 0\leq j<k},$$  
when $B$ is not the lower right box and define $\inn(T,B)$ to be the  $k$-element multiset consisting of all $k$'s when $B$ is the lower right box.  
This is the multiset of all values that slide into $B$, after they have been decremented by one. 

\begin{example}
Using the tableau $T$ from Example~\ref{ex:promotion_path} we have the following progression:
\begin{align*}
T =  \ytableaushort{125,347,689} &\longrightarrow \ytableaushort{134,268,579}\longrightarrow \ytableaushort{123,457,689}\longrightarrow \ytableaushort{126,348,579}\longrightarrow \ytableaushort{135,267,489}\longrightarrow \ytableaushort{124,356,789}\\
&\longrightarrow \ytableaushort{135,248,679}\longrightarrow \ytableaushort{124,367,589}
\longrightarrow \ytableaushort{136,258,479}\longrightarrow T.
\end{align*}
If $B$ is the box in the upper-right corner then 
$$\inn(T,B) = \{6,5,6\}\quad\textrm{and}\quad \out(T,B) = \{3,4,4\}.$$
Moreover, observe that 
$$\dis(T,B) = \{6,5,4,3\} \du \{5,4\} \du \{6,5,4\},$$ 
where the largest element of each interval is from the set $\inn$ and the smallest is from the set $\out$.  \qed
\end{example}

We now have the following relationship between the sets $\inn$ and $\out$.  

\begin{lemma}\label{lem:main}
As multisets,  $\out(T,B) = \out(\evac( T),B)$ and $\inn(T,B) = \inn(\evac( T),B)$.
\end{lemma}

Assuming Lemma~\ref{lem:main} for the moment, let us complete our alternative proof of Lemma~\ref{lem:main_gd}.  We define $[a,b] = \{a,a+1,\ldots ,b\}$ when $a\leq b$, and $[a,b] = \emptyset$ otherwise.  

\begin{proof}[Jeu de Taquin Proof of Lemma~\ref{lem:main_gd}]
It is clear from the definitions that for each value $b\in \inn(T,B)$, the following must occur at some point during our progression~\eqref{eq:promotionprogression}.  First, $b+1$ slides into box $B$ and is then decremented by one.  Then, for some number of steps $s\geq 0$ in our progression the value in $B$ does not slide and is only decremented by one for each of the $s$ steps.  Lastly, on the $(s+1)$st step the value $b-s$ slides out of box $B$; so $b-s\in\out(T,B)$.  This defines a correspondence so that every $b_i\in \inn(T,B)$ is paired with some $a_i\in \out(T,B)$ where $a_i = b_i - s_i$ for some $s_i\geq 0$.  Consequently, 
$$\dis(T,B) = [a_1,b_1]\du [a_2,b_2]\du\cdots\du  [a_r, b_r],$$ 
and the multiset $\{a_1,\ldots, a_r\} = \out(T,B)$. Now we claim that for $x \in [1,k]$, the multiplicity of $x$ in $\dis(T,B)$ is completely determined by the multisets $\out(T,B)$ and $\inn(T,B)$.  To see this define $\mathcal{A} = \makeset{a\in \out(T,B)}{a > x}$ and $\mathcal{B} = \makeset{b\in \inn(T,B)} {b\geq x}$. Since $a_i\leq b_i$, it is clear that every element in $\mathcal{A}$ must correspond to some element in $\mathcal{B}$.  But then the $|\mathcal{B}|-|\mathcal{A}|$ other elements in $\mathcal{B}$ must correspond to elements in $\out(T,B)\setminus \mathcal{A}$.  So, $|\mathcal{B}|-|\mathcal{A}|$ is the multiplicity of $x$ in $\dis(T,B)$.  By Lemma~\ref{lem:main} the same argument shows that $|\mathcal{B}|-|\mathcal{A}|$ must  also be the multiplicity of $x$ in $\dis(\evac(T),B)$, completing the proof.


\end{proof}

We now conclude this section with a proof of Lemma~\ref{lem:main}.  To begin let us return to the progression~\eqref{eq:promotionprogression}.  Imagine that at the beginning of this progression each box in $T$ contains a marker labeled with that box's value. (We think of the marker as a poker chip that can slide among the boxes in $T$.)   Now consider the marker with label $k$ in the lower right box in $T$.  Under each successive step in~(\ref{eq:promotionprogression}), this marker either remains in place, slides one unit left or one unit up and, regardless, we decrement its  label by 1.  Consequently, after $k-1$ steps this marker is labeled 1 and therefore is located in the upper left box of $\pro^{k-1} (T)$.  Consequently, this marker must have traveled along a sequence of $\ell=m+n-1$ boxes $B_1,\ldots, B_\ell$ that starts at the lower right box and ends at the upper left box.  Defining a sequence of labels $\beta_1, \ldots, \beta_\ell$ so that $\beta_1 = k$, $\beta_\ell = 1$,  and  $\beta_i$ is the label on our marker as it slides out of  $B_i$ we obtain another labeled path $\tau(T) = (B_1,\ldots, B_\ell; \beta_1,\ldots, \beta_\ell)$ called the \emph{trajectory}.  See Example~\ref{ex:trajectory} for an example.

\begin{example}\label{ex:trajectory}
For the  tableau in Example~\ref{ex:promotion_path}, we have the progression
\begin{align*}
T=\ytableaushort{125,347,68{*(gray)9}}\, &\rightarrow \ytableaushort{134,26{*(gray)8},579}\, \rightarrow \ytableaushort{123,45{*(gray)7},689}\, \rightarrow \ytableaushort{12{*(gray)6},348,579}\, \rightarrow \ytableaushort{13{*(gray)5},267,489}\, \rightarrow \ytableaushort{12{*(gray)4},356,789}\\ 
&\rightarrow \ytableaushort{1{*(gray)3}5,248,679}\, \rightarrow \ytableaushort{1{*(gray)2}4,367,589}\, \rightarrow \ytableaushort{{*(gray)1}36,258,479}\, ,
\end{align*}
where the shaded box represents the position of our marker at each step.  The trajectory is then
$$\tau(T) = \{(3,3), (2,3), (1,3), (1,2), (1,1); 9,7,4,2,1\},$$
where the boxes are indexed using matrix coordinates.  
\end{example}

\begin{remark}
Recall that promotion shifts the markers northwest and decrements their labels by one whereas  $\pro^{-1}$ does the opposite.  As a result, for each marker $M$ in $T$ there is some unique index $0\leq j<k$ so that $M$ is in the bottom right corner of $\pro^{j}(T)$.  As a result we may think of $\tau(\pro^{j}(T))$ as the ``trajectory" of the marker $M$.  
\end{remark}

\begin{lemma}\label{lem:traj=prompath}
We have $\tau(\pro^{-i}(T)) = \rho(\pro^i (\evac( T)))$.
\end{lemma}
\begin{proof}
The case when $i=0$ follows directly from the proof of \cite[Theorem 2.3]{stanley}, after modification to account for our labels.  Since $\pro^i (\evac(T)) = \evac(\pro^{-i}( T))$, the result for $i\neq 0$  follows by replacing $T$ by $\pro^{-i}( T)$ in the $i=0$ case.
\end{proof}

\begin{proof}[Proof of Lemma~\ref{lem:main}]
Fix $a\in \out(T,B)$.  In terms of markers, this means that there is a marker $M$ and an index $j$ such that  during the step $\pro^j (T) \longrightarrow \pro^{j+1} (T)$, marker $M$  slides out of box $B$ with label $a$.  By the remark above, we see that $\tau(\pro^{j-(k-a)}(T)) =(B_1,\ldots, B_\ell, \beta_1,\ldots, \beta_\ell)$ is the trajectory of the marker $M$. So for some $i$ we must have $B_i = B$ and $\beta_i = a$.  By Lemma~\ref{lem:traj=prompath} we may conclude that $a\in \out(\evac( T ),B)$. 

This argument certainly shows that $\out(T,B) \subset \out(\evac(T),B)$ as sets. Interchanging $T$ and $\evac(T)$ we have equality as sets.   To see that we actually have equality as multisets note that any two occurrences of $a$ in $\out(T,B)$ must slide out of $B$ at different points along the progression~(\ref{eq:promotionprogression}).

The second claim now follows by a few symmetries. If for a multiset $A$ we define $\overline{A}$ to be the multiset obtained by replacing each member $x$ of $A$ by $k+1-x$ then
\begin{align*}
\inn(T,B) &= \overline{\inn_{\pro^{-1}}(\evac(T), B^*)}= \overline{\out(\evac(T), B^*)} = \overline{\out(T, B^*)}= \overline{\inn_{\pro^{-1}}(T,B^*)} = \inn(\evac(T),B),
\end{align*}
where $B^*$ is the box corresponding to $B$ under $180^\circ$ rotation and $\inn_{\pro^{-1}}$ is defined the same as $\inn$ except that we replace $\pro$ by $\pro^{-1}$.  

\end{proof}

\section{Linear extensions of cominuscule posets}\label{sec:cominuscule}
In this paper, all posets will be assumed finite. A \emph{linear extension} of a poset $P$ is an order-preserving bijection onto a chain ${\bf d} := 1 < 2 < \dots < d$, where $d = |P|$. Observe that standard Young tableaux of shape $\lambda$ may be identified with linear extensions of $\lambda$, where we think of $\lambda$ as a poset in which each box is covered by those immediately below it and to its right. We write $\syt(P)$ for the set of linear extensions of a poset $P$. (Note that we do not generally have a notion corresponding to \emph{semistandard} tableaux.) There are analogous definitions of promotion and evacuation in this setting (cf.~\cite{stanley}), which we denote by $\pro$ and $\evac$ respectively. If $S$ is a set of elements in the poset $P$ and $T : P \to \bf{d}$ is a linear extension, we define similarly to before: \[ \sigma_S(T) := \sum_{s \in S} T(s).\]

We now prove a generalization of Theorem~\ref{thm:P--R} to the larger class of cominuscule posets. Although we define this class algebraically, it may also be described purely combinatorially. In defining this class of posets, we mostly follow the notation and exposition of \cite{thomas.yong:cominuscule}. We recommend \cite{billey.lakshmibai} and \cite{thomas.yong:cominuscule} for further details and references regarding these well-studied posets and associated geometry. 

Let $G$ be a complex connected reductive Lie group with maximal torus $T$. Let $W$ denote the Weyl group $N(T)/T$. Let $\Phi = \Phi^+ \sqcup \Phi^-$ denote the root system of $G$, as partitioned into positive and negative roots, with $\Delta$ denoting the choice of simple roots. The set $\Phi^+$ of positive roots has a poset structure $(\Phi^+, <)$ defined as the transitive closure of the covering relation $\alpha \lessdot \beta$ if and only if $\beta - \alpha \in \Delta$.

We say a simple root $\mu$ is \emph{cominuscule} if for every $\alpha \in \Phi^+$, $\mu$ appears with multiplicity at most 1 in the simple root expansion of $\alpha$. 
For $\mu$ cominuscule, let $\Lambda_\mu \subseteq (\Phi^+, <)$ be the subposet of positive roots for which $\mu$ appears in the simple root expansion. 
We call such a poset \emph{cominuscule}. These posets govern much of the geometry of the so-called \emph{cominuscule varieties}, which ``next to $\mathbb{P}^n$, may be considered as the simplest examples of projective varieties'' \cite[$\mathsection 9$]{billey.lakshmibai}. 
In the case $G = GL_n(\CC)$, every simple root is cominuscule, the corresponding cominuscule varieties are complex Grassmannians, and the corresponding cominuscule posets are rectangles.

The cominuscule posets are completely classified: there are three infinite families (rectangles, shifted staircases, propellers) and two exceptional examples. These are all illustrated in Figure~\ref{fig:cominuscule_posets}.

\begin{figure}[h]
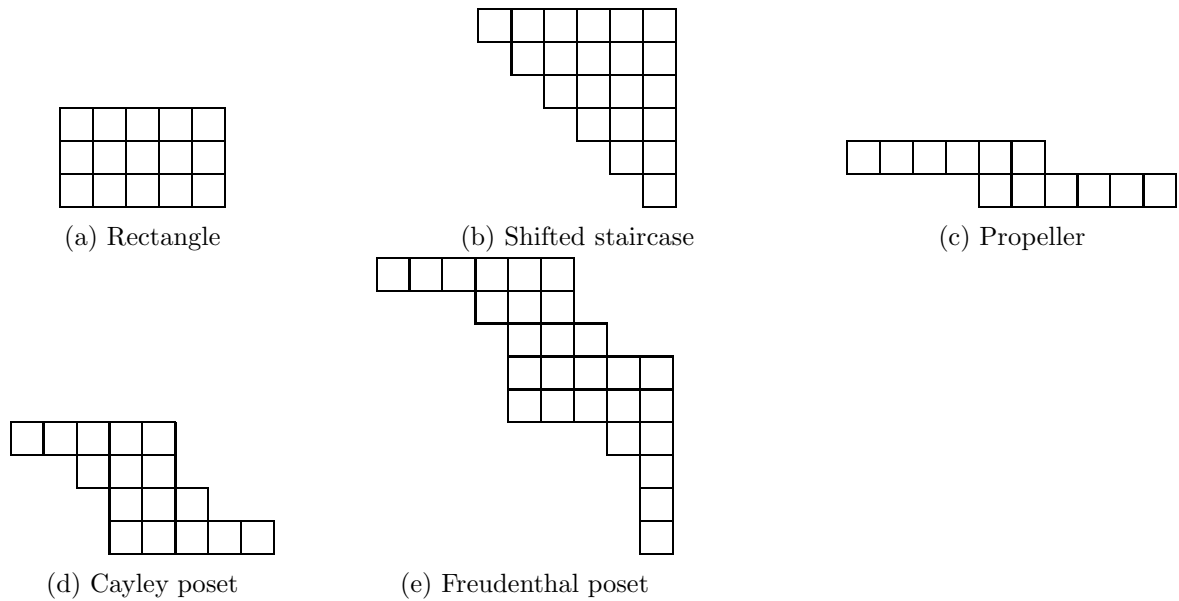

	\begin{subfigure}[b]{0.3\textwidth}
		\centering
		\ydiagram{5,5,5}
		\caption{Rectangle}
	\end{subfigure}
	\begin{subfigure}[b]{0.3\textwidth}
		\centering
		\ydiagram{6,1+5,2+4,3+3,4+2,5+1}
		\caption{Shifted staircase}
	\end{subfigure}
	\begin{subfigure}[b]{0.3\textwidth}
		\centering
		\ydiagram{6,4+6}
		\caption{Propeller}
	\end{subfigure}
	\begin{subfigure}[b]{0.3\textwidth}
		\centering
		\ydiagram{5,2+3,3+3,3+5}
		\caption{Cayley poset}
	\end{subfigure}
	\begin{subfigure}[b]{0.3\textwidth}
		\centering
		\ydiagram{6,3+3,4+3,4+5,4+5,7+2,8+1,8+1,8+1}
		\caption{Freudenthal poset}
	\end{subfigure}
\caption{The five families of cominuscule posets. The boxes are the elements of the poset, and each box is covered by any box immediately below it or immediately to its right. Rectangles may have arbitrary height and width. Shifted staircases have arbitrary width, and height equal to their width; hence a shifted staircase of width $n$ contains $\binom{n+1}{2}$ elements. Propellers consist of two rows of arbitrary but equal length, overlapping by two boxes in the center. The Cayley and Freudenthal posets are unique, containing 16 and 27 elements, respectively.}\label{fig:cominuscule_posets}
\end{figure}

The parabolic subgroups of $W$ are in canonical bijection with the subsets of $\Delta$. For $\mu$ cominuscule, let $w_\mu$ denote the longest element of the parabolic subgroup $W_\mu \leq W$ corresponding to the subset $\Delta \backslash \{ \mu \}$. It is not hard to show that $w_\mu$ acts as an involution on $\Lambda_\mu$. Following \cite[$\mathsection 2.2$]{thomas.yong:cominuscule}, we denote this action on $\Lambda_\mu$ by $\mathtt{rotate}$. For rectangles, propellers, and the Cayley poset, this action is exactly $180^\circ$ rotation. For shifted staircases and the Freudenthal poset, it is reflection across the antidiagonal. 

The following theorem generalizes Theorem~\ref{thm:P--R} to include nonrectangular cominuscule posets.

\begin{theorem}\label{thm:cominuscule}
Let $P$ be a cominuscule poset, $S \subseteq P$ a set of elements fixed under $\mathtt{rotate}$, and $\mathcal{C} = \langle c \rangle$, the cyclic group with $c$ acting on $\syt(P)$ by promotion. Then \[(\syt(P), \mathcal{C}, \sigma_S)\] is homomesic.
\end{theorem}
\begin{proof}
Let $T \in \syt(P)$ with $P$ cominuscule. By \cite[Lemma~5.2]{thomas.yong:ctc}, $\evac(T)$ may be formed by applying $\mathtt{rotate}$ and reversing the alphabet (so $i$ becomes $|P| + 1 - i$). 

The theorem then follows from a poset analogue of Lemma~\ref{lem:main_gd}. For this the growth diagram proof of Lemma~\ref{lem:main_gd} may be copied nearly verbatim, using the cardinality of the promotion orbit in place of $k$.
\end{proof}

\section{Increasing tableaux and K-promotion}\label{sec:K}
For posets $P_1, P_2$, we say $\phi : P_1 \to P_2$ is \emph{strictly order-preserving} if $x < y$ implies $\phi(x) < \phi(y)$. Let $P$ be a finite poset. An \emph{increasing tableau} (of shape $P$) is a strictly order-preserving surjection $T : P \to {\bf d}$, where ${\bf d} := 1 < 2 < \dots < d$ and $d$ is potentially less than $|P|$. In the case where $P$ is a Ferrers poset, such tableaux may be realized as semistandard tableaux such that all rows and columns are strictly increasing and the set of entries is an initial segment of $\ZZ_{> 0}$. We denote by $\inc{P}{q}$ the set of increasing tableaux $T : P \to \bf{d}$ where $d = |P|-q$. Notice $\inc{P}{0} = \syt(P)$. 

Increasing tableaux perhaps appeared first in relation to the Edelman--Greene correspondence \cite{edelman.greene}, and later are implicit in early work on Aztec diamond tilings \cite{jockusch.propp.shor}. More recently their combinatorics has been further developed by H.~Thomas and A.~Yong \cite{thomas.yong:K} and others in service of K-theoretic Schubert calculus. (The definition of increasing tableaux in \cite{thomas.yong:K} is slightly more general, in that $T$ is not required to be surjective. However \cite{thomas.yong:K} makes no use of this additional generality, and in light of the enumerative results of \cite{pechenik}, we believe the more restrictive definition given here is of greater interest.)

Following \cite{thomas.yong:K, pechenik}, we define an operation of K-promotion on $\inc{P}{q}$. For $T \in \inc{P}{q}$ and $p \in P$, we think of $T(p)$ as a label on the element $p$. For every pair of labels $\{a, b\}$, we define an operation $\mathtt{switch}_{a, b}$. Every element labeled $a$ is relabeled $b$ if it covers or is covered by an element labeled $b$. (For tableaux, this is equivalent to the labels $a,b$ appearing in adjacent boxes.) Simultaneously, every element labeled $b$ is relabeled $a$ if it covers or is covered by a element labeled $a$. (The result of this operation will not generally be an increasing tableau.)

The K-promotion of $T \in \inc{P}{q}$ is formed as follows. First replace each label 1 with a bullet~$\bullet$. Then for $2 \leq i \leq |P| - q$, successively apply the operators $\mathtt{switch}_{i,\bullet}$. Finally reduce each label by 1, and replace each bullet with the label $|P| - q$. We denote the operation of K-promotion by $\pro_K$.

\begin{example}\label{ex:K-promotion}
Let \[ T = \ytableaushort{13,24,45} \] be an increasing tableau. Then the K-promotion $\pro_K(T)$ is contructed by the following sequence of modifications:

\[\ytableaushort{13,24,45} \longrightarrow \ytableaushort{\bullet 3, 24,45} \longrightarrow \ytableaushort{23,\bullet 4,45} \longrightarrow \ytableaushort{23,\bullet 4,45} \longrightarrow \ytableaushort{23,4 \bullet,\bullet 5} \longrightarrow \ytableaushort{23,45,5 \bullet} \longrightarrow \ytableaushort{12,34,45}\] \qed
\end{example}

Note that although when $P$ is a Ferrers poset increasing tableaux may be realized as a subclass of semistandard tableaux, K-promotion of such an increasing tableau $T$ generally differs from its promotion as a semistandard tableau. The two concepts, however, agree in the case where $T$ is in fact standard.

\begin{example} 
Let $T$ be the tableau of Example~\ref{ex:K-promotion}, {\bf thought of as a semistandard tableau}. Then its promotion is $\ytableaushort{12,33,45}$. Notice that $\pro(T)$ is here not even an increasing tableau. \qed
\end{example}

For $T \in \inc{P}{q}$, let $T_{\leq j}$ denote the order ideal consisting of those $p \in P$ with $T(p) \leq j$. As previously discussed for semistandard tableaux, $T$ may be encoded by the chain of order ideals $(T_{\leq j})_{0 \leq j \leq |P|-q}$. In \cite[$\mathsection 4$]{thomas.yong:K}, H.~Thomas and A.~Yong define the K-evacuation $\evac_K(T)$ of an increasing tableau $T$ to be given by the chain of order ideals $(\pro_K^{|P|- q - j}(T)_{\leq j})_{0 \leq j \leq |P|-q}.$ 

\begin{example}
For the $T$ of Example~\ref{ex:K-promotion}, $\evac_K(T) = \ytableaushort{12,24,35}$, as it is encoded by the chain \[\left(\ytableausetup{smalltableaux} \emptyset - \ydiagram{1} - \ydiagram{2,1} - \ydiagram{2,1,1} - \ydiagram{2,2,1} - \ydiagram{2,2,2}\right) \ytableausetup{nosmalltableaux}.\] \qed
\end{example}

One might hope for Theorem~\ref{thm:cominuscule} to generalize to $\inc{P}{q}$ for any $q$ and any cominuscule poset $P$. However this is not the case:

\begin{example}\label{ex:increasing_counterexample}
Consider
\[ T = \ytableaushort{1235,2{*(black)\textcolor{white}{4}}{*(black)\textcolor{white}{5}}7,3689} \hspace{.5in} \text{and} \hspace{.5in} U = \ytableaushort{1456,2{*(black)\textcolor{white}{6}}{*(black)\textcolor{white}{7}}8,3789}\] and let $S$ be the $\mathtt{rotate}$-fixed set of black boxes. The K-promotion orbits $\mathcal{O}_T, \mathcal{O}_U$ of $T, U$ respectively are both of size 9. However
it may be computed that \[ \frac{1}{9} \sum_{A \in \mathcal{O}_T} \sigma_S(A) = \frac{91}{9}, \text{ while } \frac{1}{9} \sum_{B \in \mathcal{O}_U} \sigma_S(B) = 10.\]
\qed
\end{example}

Say a pair $(P,q)$ (with $P$ a cominuscule poset and $0 \leq q \leq |P|$) is \emph{homomesic} if for any $S \subseteq P$ fixed under $\mathtt{rotate}$, $(\inc{P}{q}, \mathcal{C}, \sigma_S)$ is homomesic. It seems an interesting question to classify all homomesic pairs $(P,q)$. Theorem~\ref{thm:cominuscule} shows that $(P,0)$ is homomesic for all $P$. Example~\ref{ex:increasing_counterexample} shows that $({\bf 3} \times {\bf 4}, 3)$ is \emph{not} homomesic. The following theorem shows, however, that $({\bf 2} \times {\bf n}, q)$ is always homomesic.

\begin{theorem}\label{thm:2row_K}
Let $P$ be a $2 \times n$ rectangle for any $n$, and let  $S \subseteq P$ be a set of elements fixed under $180^\circ$ rotation. Then for any $q$, $(\inc{P}{q}, \mathcal{C}, \sigma_S)$ is homomesic.
\end{theorem}
\begin{proof}
By \cite[Theorem~1.3]{pechenik}, the order of $\pro_K$ on $\inc{P}{q}$ is $2n-q$. The fact that in this case K-evacuation is the same as $180^\circ$ rotation plus alphabet reversal is \cite[Proposition~3.9]{pechenik}. The theorem then follows from an analogue of Lemma~\ref{lem:main_gd}. For this the growth diagram proof of Lemma~\ref{lem:main_gd} may be copied nearly verbatim, using K-growth diagrams and replacing every instance of $k$ in that proof with $2n-q$ (the order of $\pro_K$).

The construction of K-growth diagrams for increasing tableaux is exactly analogous to the construction of growth diagrams for semistandard tableaux; in the rows of the diagram, one merely writes the chains in Young's lattice that encode successive K-promotions, instead of promotions. It is easy to show then that Theorem~\ref{thm:coxeter_rltn}(a) and (b) hold also for $\evac_K$ and $\pro_K$. More details and examples of this construction appear in \cite[$\mathsection 3$]{pechenik} and \cite[$\mathsection2,4$]{thomas.yong:K}.
\end{proof}

\section*{Acknowledgements}
OP was supported during this project by a NSF Graduate Research Fellowship and NSF grant DMS 0838434 EMSW21MCTP: Research Experience for Graduate Students. 

The authors are grateful to Darij Grinberg, Tom Roby, and Dominic Searles for helpful conversations, and to Hugh Thomas and Alexander Yong for directing them to \cite{thomas.yong:ctc} for a characterization of cominuscule evacuation. We also thank David Speyer for helpful comments on an earlier draft.  

\bibliographystyle{siam}
\hspace*{1cm}
\bibliography{homomesy}

\end{document}